\documentclass[a4paper,12pt]{amsart}
\usepackage{amssymb}
\usepackage{xypic}
\usepackage[all,cmtip]{xy}
\usepackage{amsmath}
\usepackage{chemarrow}
\usepackage[colorlinks, 
pdfborder=001, 
linkcolor=cyan,
anchorcolor=magenta
citecolor=green
]{hyperref}
\usepackage{mathrsfs}
\usepackage{wasysym}
\usepackage{titletoc}
\usepackage{bm}
\usepackage{bbm}
\usepackage{bbold} 
\usepackage{cite}
\usepackage{subfiles}
\usepackage{extarrows} 
\usepackage{geometry}
\usepackage{colordvi}
\usepackage{color}
\usepackage{stmaryrd}

\allowdisplaybreaks

\DeclareMathOperator{\Aut}{Aut}

\DeclareMathOperator{\Coker}{Coker}
\DeclareMathOperator{\D}{\mathcal{D}}

\DeclareMathOperator{\End}{End}
\DeclareMathOperator{\Ext}{Ext}

\DeclareMathOperator{\GKdim}{GKdim}

\DeclareMathOperator{\gld}{gldim}

\DeclareMathOperator{\GrAut}{GrAut}

\DeclareMathOperator{\GrMod}{GrMod}

\DeclareMathOperator{\hdet}{hdet}

\DeclareMathOperator{\Hom}{Hom}

\DeclareMathOperator{\id}{id}

\DeclareMathOperator{\Kdim}{Kdim}

\DeclareMathOperator{\kk}{\Bbbk}

\DeclareMathOperator{\m}{\mathfrak{m}}

\DeclareMathOperator{\N}{\mathbb{N}}

\DeclareMathOperator{\op}{op}
\DeclareMathOperator{\p}{\mathfrak{p}}

\DeclareMathOperator{\RHom}{\textbf{R}Hom}
\DeclareMathOperator{\rhu}{\rightharpoonup}

\DeclareMathOperator{\SL}{\mathrm{SL}}
\DeclareMathOperator{\Spec}{Spec}

\DeclareMathOperator{\Tr}{Tr}
\DeclareMathOperator{\vep}{\varepsilon}

\DeclareMathOperator{\Z}{\mathbb{Z}}

\newcommand{\To}{\longrightarrow}

\numberwithin{equation}{section}

\theoremstyle{definition}

\newtheorem{thm}{Theorem}[section]
\newtheorem{prop}[thm]{Proposition}
\newtheorem{lem}[thm]{Lemma}

\newtheorem{defn}[thm]{Definition}

\newtheorem{ex}[thm]{Example}

\newtheorem{conj}[thm]{Conjecture}

\begin{document}

\title{Auslander theorem for PI Artin-Schelter regular algebras}

\author{Ruipeng Zhu}
\address{School of Mathematics, Shanghai University of Finance and Economics, Shanghai 200433, China}
\email{zhuruipeng@sufe.edu.cn}

\begin{abstract}
	We prove a version of a theorem of Auslander for finite group actions or coactions on noetherian polynomial identity Artin-Schelter regular algebra.
\end{abstract}



\maketitle



\section*{Introduction}

Throughout, $\kk$ is an algebraically closed field of characteristic zero. All vector spaces, algebras and Hopf algebras are over $\kk$.

A classical theorem of Maurice Auslander 
states that if $G$ is a finite subgroup of $\SL_n(\kk)$, acting linearly on the commutative polynomial ring $R = \kk[x_1, \dots, x_n]$ with invariant subring $R^G$, then the natural map from the skew group algebra $R\#G$ to $\End_{R^G}(R)$ is an isomorphism of graded algebras.
The Auslander theorem is a fundamental result in the study of the McKay correspondence.
Recently, some researchers have studied the Auslander theorem and McKay correspondence in the noncommutative setting, see \cite{BHZ2018, BHZ2019, CKWZ2018, CKWZ2019, CKZ2020, GKMW2019, Mor2013}.
One of the main open questions concerning a noncommutative version of Auslander theorem is the following conjecture that was stated in \cite[Conjecture 0.2]{CKWZ2018}.

\begin{conj}
	Let $R$ be a noetherian connected graded Artin-Schelter regular algebra \cite{AS1987} and $H$ be a semisimple Hopf algebra acting homogeneously and inner faithfully on $R$. 
	If the homological determinant of the $H$-action on $R$ is trivial, then the Auslander map
	$$\varphi: R\#H \To \End_{R^H}(R), \qquad r\#h \mapsto \big(x \mapsto r (h \rhu x)\big)$$
	 is an isomorphism of graded algebras.
\end{conj}

It is worth pointing out that the Hopf action with trivial homological determinant generalizes the group action by a subgroup of $\SL_n(\kk)$.

The above conjecture holds when $R$ has global dimension two \cite[Theorem 0.3]{CKWZ2018}.
Bao, He and Zhang proved that the Auslander map is related to an invariant of the $H$-action on $R$ known as the pertinency.
The {\it pertinency} of the $H$-action on $R$ is defined to be
$$\mathrm{p}(R, H) = \GKdim R - \GKdim R\#H / (1\#t)$$
where $t$ is a nonzero integral of $H$, see \cite[Definition 0.1]{BHZ2019}.

\begin{thm}\cite[Theorem 0.3]{BHZ2019}
	Let $R$ be a noetherian, connected graded Artin-Schelter regular, Cohen-Macaulay algebra of Gelfand-Kirillov dimension at least two. Let $H$ be a semisimple Hopf algebra acting on $R$ homogeneously and inner faithfully. Then the Auslander map is an isomorphism if and only if $\mathrm{p}(R, H) \geq 2$.
\end{thm}

Some other interesting partial results concerning Auslander's Theorem have been proved in \cite{BHZ2018, BHZ2019, GKMW2019, CKZ2020} by using the pertinency as a major tool.

The goal of this paper is to verify the conjecture for finite group actions and coactions on polynomial identity (PI) Artin-Schelter regular algebras. The ideal of the proof is to use a result given by Buchweitz \cite[Proposition 2.9]{Buc2003}.

\begin{thm}(Theorem \ref{Auslander-thm-for-PI-rings}) \label{Auslander-thm-for-PI-rings'}
	Let $R$ be a noetherian PI Artin-Schelter regular algebra, and $H$ be a semisimple Hopf algebra acting on $R$ homogeneously and inner faithfully. 
	Suppose that $H = \kk G$ or $(\kk G)^*$ where $G$ is a finite group.
	If the homological determinant of the $H$-action on $R$ is trivial, then the Auslander map $\varphi: R\#H \to \End_{R^H}(R)$ is bijective.
\end{thm}

In the paper \cite{QWZ2019}, Qin, Wang and Zhang shows that whenever Auslander Theorem holds one can view $R\#H$ as a noncommutative quasi-resolution (NQR for short, which is a generalization of noncommutative crepant resolution (NCCR) in the sense of Van den Bergh \cite{VdBer2004}) of $R^H$, and when $R^H$ is a central subalgebra of $R$, $R\#H$ is a NCCR of $R^H$.

Applying Theorem \ref{Auslander-thm-for-PI-rings'}, we have the following result concerning whenever the center of a noetherian PI Artin-Schelter regular algebra has a NCCR.

\begin{thm}[Theorem \ref{NCCR-thm}]
	Let $R$ be a noetherian PI connected graded Calabi-Yau algebra with the center $Z$. If $R^{\GrAut_{Z}(R)} = Z$, then $\End_Z(R)$ is a NCCR of $Z$.
\end{thm}

\section{Noncommutative Auslander Theorem}

\subsection{The Auslander map for $(B, e)$}

Firstly, we recall some definitions.

\begin{defn}\label{defn-Auslander-Gorenstein ring}
	Let $R$ be a (left and right) noetherian algebra.
	
	(1) The {\it grade} of a left or right $R$-module $M$ is
	$$j_R(M) = j(M) := \inf \{ j \mid \Ext^j_R(M, R) \neq 0 \}.$$
	
	(2) $R$ is called to satisfy the {\it Auslander condition} if, for every finitely generated left or right $R$-module $M$ and for all $i \geq 0$, $j_R(N) \geq i$ for all right or left submodules $N \subseteq \Ext^i_R(M,R)$.
	
	(3) $R$ is called {\it Auslander Gorenstein} if it satisfies the Auslander condition, and has finite left and right injective dimensions. If further, $R$ has finite global dimension, then $R$ is called {\it Auslander regular}.
	
	(4) $R$ is called {\it Cohen-Macaulay} (or {\it CM} for short) with respect to Gelfand-Kirillov (or GK for short) dimension and Krull dimension if for any finitely generated $R$-module $M$
	$$j(M) + \GKdim M = \GKdim R < + \infty \; \text{ and } \; j(M) + \Kdim M = \Kdim R < + \infty,$$
	respectively. See \cite{MR2001} for the definitions of GK-dimension and Krull dimension.
\end{defn}

\begin{lem}\label{Auslander-Gorenstein-lem}
	Let $R$ be a noetherian Auslander Gorenstein ring. If $0 \to M' \to M \to M'' \to 0$ is an exact sequence of finitely generated $R$-modules, then $j(M) = \min \{ j(M'), j(M'') \}$.
\end{lem}

An algebra $R$ is called {\it connected graded} if $R = \oplus_{i \in \N} R_i$, $1 \in R_0 = \kk$, and $R_iR_j \subseteq R_{i+j}$ for all $i, j \in \N$.
A {\it polynomial identity algebra} (or {\it PI algebra} for short), is an algebra satisfying a polynomial identity. We refer to \cite[Chapter 13]{MR2001} for some basic materials about PI algebras.

\begin{lem}\cite[Lemma 6.1]{SZ1994}
	Let $R$ be a connected graded PI ring $R$ and $M$ be a finitely generated graded $R$-module. Then $\mathrm{GKdim} M = \mathrm{Kdim} M \in \N$.
\end{lem}

Throughout this section, let $B$ be a ring and $e \in B$ be a nonzero idempotent. The {\it Auslander map} is the natural map
$$\varphi: B \to \End_{eBe}(Be), \text{ which is defined by } \varphi(b)(x) = bx \text{ for all } b \in B, x \in Be.$$

In order to show our main theorem, we need the following results proved by Buchweitz.

\begin{prop}\cite[Proposition 2.9]{Buc2003}\label{grade-Auslander-map}
	\begin{enumerate}
		\item $\varphi$ is injective if and only if $j(B/BeB) \geq 1$,
		\item $\varphi$ is bijective if and only if $j(B/BeB) \geq 2$.
	\end{enumerate}
\end{prop}

Let $A$ be a $\kk$-algebra and $M$ be a right $A$-module. Then there is a morphism
$$\gamma_M: M \otimes_A \Hom_{A}(M, A) \To \End_A(M), \qquad x \otimes f \mapsto \big(y \mapsto xf(y) \big).$$
Set $A = eBe$. Then there is an $A$-$B$-bimodule morphism
$$\nu: eB \to \Hom_A(Be, A), \qquad y \mapsto (x \mapsto yx).$$
Now we have a useful commutative diagram of $B$-$B$-bimodule morphisms
\begin{equation}\label{com-dig}
\begin{split}
	\xymatrix{Be \otimes_{A} eB \ar[rr]^-{\mu} \ar[d]_{\id \otimes \nu} && B \ar[d]^{\varphi} \\
		Be \otimes_{A} \Hom_A(Be, A) \ar[rr]^-{\gamma = \gamma_{\small Be}} && \End_{A}(Be). }
\end{split}
\end{equation}
where $\mu(x \otimes y) = xy$ for any $x \in Be$ and $y \in eB$.

\begin{lem}\label{nu-injective}
	If the Auslander map $\varphi$ is injective, then the morphism $\nu$ is also injective.
\end{lem}
\begin{proof}
	The conclusion follows essentially from the commutativity of the following diagram
	\begin{equation*}
		\xymatrix{
			eB \ar@{>->}[r] \ar[d]^{\nu} & B \ar@{>->}[d]^{\varphi} \\
			\Hom_A(Be, A) \ar@{>->}[r] & \End_A(Be).
		}
	\end{equation*}
\end{proof}

\begin{prop}\label{main-prop-0}
	If the Auslander map $\varphi: B \to \End_{A}(Be)$ is bijective, then
	\begin{enumerate}
		\item[(a)] $\varphi$ is injective,
		\item[(b)] $\nu$ is bijective,
		\item[(c)] $j_B(\Coker \gamma) \geq 2$.
	\end{enumerate}
\end{prop}
\begin{proof}
	(a) It is obvious. 
	
	(b) 
	By Lemma \ref{nu-injective}, we only need to show that $\nu$ is surjective.
	For any $f \in \Hom_A(Be, A)$, there exists an element $b \in B$ such that $\varphi(b) = f$ because $\varphi$ is bijective. Since $\nu(eb) = \varphi(eb) = e\varphi(b) = ef = f$, it follows that $\nu$ is surjective as required.
	
	(c) 
	According to the commutativity of the diagram \eqref{com-dig}, $\Coker \gamma \cong \Coker \mu = B/BeB$. Hence by Proposition \ref{grade-Auslander-map}, $j_B(\Coker \gamma) \geq 2$.
\end{proof}

By using the Auslander property, we show that the necessary condition for the bijectivity of the Auslander map in Proposition \ref{main-prop-0}, is also sufficient.

\begin{thm}\label{main-thm-0}
	Suppose that $B$ is a noetherian Auslander Gorenstein ring. Then $\varphi$ is bijective if and only if the condition (a), (b) and (c) in Proposition \ref{main-prop-0} hold.
\end{thm}
\begin{proof}
	It suffices to show that $\varphi$ is bijective when the condition (a), (b) and (c) in Proposition \ref{main-prop-0} hold.
	Since $\nu$ is bijective and $\varphi$ is injective, it follows from the commutative diagram \eqref{com-dig} that the sequence
	$$0 \To \Coker \mu = B/BeB \To \Coker \gamma \To \Coker \varphi \To 0$$
	is exact.
	According to Lemma \ref{Auslander-Gorenstein-lem}, $j_B(B/BeB) \geq j_B(\Coker \gamma) \geq 2$.
	Hence $\varphi$ is bijective by Proposition \ref{grade-Auslander-map}.
\end{proof}

\subsection{Smash products and the morphism $\widetilde{t}$}

In the following, let's consider the Hopf action.
Let $H$ stand for a Hopf algebra $(H, \Delta, \varepsilon)$ with the bijective antipode $S$. We use the Sweedler notation $\Delta(h) = \sum_{(h)} h_1 \otimes h_2$ for all $h \in H$.
We recommend \cite{Mon1993} as a basic reference for the theory of Hopf algebras and their actions on algebras.

Let $H$ be a finite dimensional semisimple Hopf algebra, and $t$ be a nonzero integral of $H$ with $\vep(t) = 1$.
Let $R$ be a left $H$-module algebra, i.e., $H$ acts on $R$. The $H$-action on $R$ is said {\it inner faithful} if $I \rhu R \neq 0$ for every nonzero Hopf ideal $I$ of $H$.

Now assume that $B$ is the smash product algebra $R\#H$, and that $e$ is the idempotent $1\#t$.
Let $R^H$ be the invariant subring of $R$ under the $H$-action. Then as $\kk$-algebras
$$A = eBe = (1\#t)(R\#H)(1\#t) = R^H \# t \cong R^H,$$
and $Be = R \# t \cong R$ as $B$-$A$-bimodules.
Thus we can rewrite the Auslander map as
$$\varphi: R\#H \To \End_{R^H}(R), \qquad r\#h \mapsto \big(x \mapsto r (h \rhu x)\big).$$

\begin{lem}
	There is a right $R$-module isomorphism $\lambda: R \to eB, \; r \mapsto \sum_{(t)} (t_1 \rhu r) \# t_2$.
\end{lem}
\begin{proof}
	Let $H^*$ be the dual Hopf algebra of $H$, and $\alpha$ be an integral of $H^*$ with $\langle \alpha, t \rangle = 1$.
	We claim that the map
	$$\lambda': eB \to R, \qquad \sum\limits_{i} r_i \# h_i \mapsto \sum\limits_{i} r_i \langle \alpha, h_i \rangle$$
	is the inverse of $\lambda$.
	Note that $\sum\limits_{(t)} t_1 \langle \alpha, t_2 \rangle = 1$.
	For any $r \in R$,
	$$\lambda'\lambda(r) = \sum t_1 \rhu r \langle \alpha, t_2 \rangle = \langle \alpha, t \rangle r = r.$$
	For any $\sum\limits_{i} r_i \# h_i \in B$,
	\begin{align*}
		\lambda\lambda'(e(\sum_{i} r_i \# h_i))
		& = \lambda\lambda'(\sum\limits_{i, (t)} t_1 \rhu r_i \# t_2h_i) = \lambda\lambda' (\sum\limits_{i, (t)} (t_1S^{-1}h_i) \rhu r_i \# t_2) \\ 
		& = \lambda (\sum\limits_{i, (t)} (t_1S^{-1}h_i) \rhu r_i \langle \alpha, t_2 \rangle) = \lambda (\sum\limits_{i} S^{-1}h_i \rhu r_i) \\
		& = \sum\limits_{i, (t)} t_1S^{-1}h_i \rhu r_i \# t_2 = \sum\limits_{i, (t)} t_1 \rhu r_i \# t_2h_i \\
		& = e(\sum_{i} r_i \# h_i).
	\end{align*}
	Hence $\lambda$ is an isomorphism.
\end{proof}

Let's consider a map $\widetilde{t}: R \to \Hom_{R^H}(R, R^{H})$ which is defined by $\widetilde{t}(r)(x) = t \rhu (rx)$.

\begin{lem}\label{nu-widetilde-t}
	Then $\nu$ is bijective if and only if $\widetilde{t}$ is bijective.
\end{lem}
\begin{proof}
	The conclusion follows from the commutative diagram
	$$\xymatrix{
		R \ar[r]^{\lambda} \ar[d]^{\widetilde{t}} & eB \ar[d]^{\nu} \\
		\Hom_{R^H}(R, R^{H}) \ar[r]^-{\cong} & \Hom_A(Be, A). \\
	}$$
\end{proof}

\subsection{Homological determinant and the bijectivity of $\nu$}


In the following, let's recall firstly the definition of Artin-Schelter regular algebras \cite{AS1987}.


\begin{defn}\label{defn-AS-Gorenstein}
	A noetherian connected graded algebra $R$ is called {\it Artin-Schelter Gorenstein} (or {\it AS Gorenstein}, for short) of dimension $d$, if the following conditions hold:
	
	(1) $R$ has finite injective dimension $d$ on the left and on the right,
	
	(2) $\Ext^i_R(\kk, R) \cong \Ext^i_{R^{\op}}(\kk, R) \cong \begin{cases}
			0, & i \neq d \\
			\kk(l), & i = d \\
		\end{cases}$, for some integer $l$, where $\kk:= R/\oplus_{i > 0} R_i$. Here $l$ is called the {\it AS index} of $R$.

	If in addition, $R$ has finite global dimension and finite GK dimension, then $R$ is called {\it Artin-Schelter regular} (or {\it AS regular}, for short) of dimension $d$.
\end{defn}


The homological determinant has been an important tool for understanding the theory of Hopf algebra actions on connected graded AS Gorenstein algebras.
Let $R$ be a connected graded AS Gorenstein algebra. Assume that $R$ is a graded left $H$-module algebra, i.e., $H$ acts on $R$ homogeneously.
The homological determinant of $H$-action on $R$ is an algebra homomorphism $\hdet: H \to \kk$ (see \cite[Definition 3.3]{KKZ2009} for the definition).
If $\hdet = \vep$, then we say that the homological determinant $\hdet$ is trivial.
The graded automorphism group of $R$ is denoted by $\GrAut(R)$, and the special linear automorphism group $\SL(R)$ is the group of graded automorphisms of $R$ with homological determinant $1$:
$$\SL(R) := \{ g \in \GrAut(R) \mid \hdet(g) = 1 \}.$$

The dualizing complexes over noncommutative rings were introduced by Yekutieli in \cite{Yek1992}.
It was studied further by Van den Bergh \cite{VdB1997} and Yekutieli-Zhang \cite{YZ1999}.

\begin{defn}
	Let $R$ be a noetherian connected graded $\N$-graded algebra.
	Let $\D(\GrMod R^e)$ be the derived category of complexes of left graded $R^e$-modules.
	A bounded complex $\Omega \in \D(\GrMod R^e)$ is called a {\it dualizing complex} over $R$ if it satisfies the following conditions:
	\begin{enumerate}
		\item $\Omega$ has finite injective dimension on both sides,
		\item each cohomology $H^i(\Omega)$ is noetherian over $R$ and over $R^{\mathrm{op}}$ for every $i$,
		\item the canonical morphisms $R \to \RHom_R(\Omega, \Omega)$ and $R \to \RHom_{R^{\mathrm{op}}}(\Omega, \Omega)$ are isomorphisms in $\D(\GrMod R^e)$.
	\end{enumerate}
\end{defn}

Let $R$ be a noetherian connected graded algebra and $\m$ be the maximal graded ideal of $R$.
For any graded right $R$-module $M$, the $\m$-torsion functor $\Gamma_{\m}$ is defined to be
$$\Gamma_{\m}(M) = \{ x \in M \mid x \m^n = 0, \exists \, n \geq 0 \}.$$
For any graded left $R$-module $M$, the $\m^{\mathrm{op}}$-torsion functor $\Gamma_{\m^{\mathrm{op}}}$ is defined to be
$$\Gamma_{\m^{\mathrm{op}}}(M) = \{ x \in M \mid \m^n x = 0, \exists \, n \geq 0 \}.$$
The derived functor $\mathbf{R}\Gamma_{\m}$ (respectively, $\mathbf{R}\Gamma_{\m^{\mathrm{op}}}$) is defined on the derived category $\D^{+}(\GrMod R)$ (respectively, $\D^{+}(\GrMod R^{\mathrm{op}})$) of bounded below complexes of graded right (respectively, left) $R$-modules.

See \cite{AZ1994} for more details.

\begin{defn}
	A dualizing complex $\Omega$ over a noetherian connected graded algebra $R$ is called {\it balanced} if there are isomorphisms
	$$\mathbf{R}\Gamma_{\m}(\Omega) \cong \mathbf{R}\Gamma_{\m^{\mathrm{op}}}(\Omega) \cong \bigoplus_{n \in \N} \Hom(R_n, \Bbbk)$$
	in $\D(\GrMod R^e)$.
\end{defn}

For any $R$-$R$-bimodule $M$ and automorphism $\sigma: R \to R$, ${^{\sigma}M}$ is the $R$-$R$-bimodule with the same ground vector space $M$ and the bimodule structure is given by $x \cdot m \cdot y = \sigma(x)my$.

\begin{lem}\cite[Lemma 1.6]{KKZ2009}
	Let $R$ be a noetherian connected graded algebra. Then $R$ is AS Gorenstein if and only if $R$ admits a balanced dualizing complex of the form $^{\mu}R(-l)[d]$ where $\mu$ is a graded algebra automorphism of $R$ and where $d$ and $l$ are the integers given in Definition \ref{defn-AS-Gorenstein}. 
\end{lem}

If $R$ admits a balanced dualizing complex $^{\mu}R(-l)[d]$, then $\mu$ is called the {\it Nakayama automorphism} of $R$.
It is clear that $\mu$ is unique up to inner automorphisms of $R$.

We need the following results about the balanced dualizing complex.

\begin{lem}\label{dualizing complex}
	Let $R$ be a noetherian connected graded AS regular algebra of dimension $d$ with AS index $l$, and $H$ be a semisimple Hopf algebra.
	Suppose that $H$ acts homogeneously and inner faithfully on $R$ such that the homological determinant is trivial.
	
	(1) \cite[Section 3]{KKZ2009} 
	Then $R^H$ is AS Gorenstein of dimension $d$ with a balanced dualizing complex $\Omega(R^H):= {^{\mu}(R^H)(-l)[d]}$, where $\mu$ is the Nakayama automorphism of $R^H$.
	
	(2) \cite[Corollary 4.17]{YZ1999}
	Then $\Omega(R) \cong \RHom_{R^H}(R, \Omega(R^H))$ in $\D(\GrMod R)$, where $\Omega(R)$ is the balanced dualizing complex of $R$.
\end{lem}

\begin{lem}\label{nu-bijective}
	Let $H$ be a semisimple Hopf algebra, and $R$ be a noetherian connected graded AS regular algebra. Suppose that $H$ acts homogeneously and inner faithfully on $R$.
	If the homological determinant is trivial, then the map $\nu$ is bijective.
\end{lem}
\begin{proof}
	According to Lemma \ref{dualizing complex}, $\Hom_{R^H}(R, R^H) \cong R$ as graded right $R$-modules.
	Then the nonzero graded $R$-module morphism $\widetilde{t}: R \to \Hom_{R^H}(R, R^H)$ is bijective.
	Thus $\nu$ is bijective by Lemma \ref{nu-widetilde-t}.
\end{proof}

\subsection{The injectivity of the Auslander map}

Under some mild assumptions, the primeness of $R\#H$ is equivalent to the faithfulness of the action of $R\#H$ on $R$, see \cite[Theorem 3.1]{BCF1990} and \cite[Lemma 3.10]{BHZ2019} for example.
Obviously, the Auslander map is injective if and only if $R$ is a faithful $R\#H$-module.

The next lemma follows immediately from \cite[Lemma 4.2]{Skr2018} and \cite[Corollary 3.4]{BCF1990}.

\begin{lem}\label{varphi-injective-lemma}
	Let $H$ be a semisimple Hopf algebra, and $R$ be a domain which is a noetherian $H$-module algebra.
	Then $R\#H$ is prime if and only if the Auslander map is injective.
\end{lem}

Let's recall some results about Galois theory of division rings. See \cite{Mon1993} for the definition of Hopf Galois extension.

\begin{thm}\cite[Theorem 3.1]{CE2017}\label{group-acts-on-division-algebra}
	Let $D$ be a division algebra, and $G$ be a finite subgroup of $\Aut(D)$. Then the extension $D^G \subseteq D$ is a $(\kk G)^*$-Galois extension.
\end{thm}


\begin{lem}\cite[Theorem A.I.4.2]{NO1982}\label{group-graded-division-algebra}
	Let $G$ be a finite group and $D = \oplus_{g \in G} D_g$ be a $G$-graded division ring.
	Suppose that $(\kk G)^*$-action on $D$ is inner faithful, that is, $D_g \neq 0$ for all $g \in G$.
	Then $D$ is strongly $G$-graded, that is, $D_1 \subseteq D$ is a $\kk G$-Galois extension.
\end{lem}

Now we can prove the following result concerning the injectivity of the Auslander map for group actions and coactions.

\begin{prop}\label{Auslander-map-injective}
	Let $R$ be a noetherian domain, and $H$ be a semisimple Hopf algebra acting on $R$ inner faithfully. Suppose that $H = \kk G$ or $(\kk G)^*$ where $G$ is a finite group. Then the Auslander map $\varphi: R \# H \to \End_{R^H}(R)$ is injective.
\end{prop}
\begin{proof}
	Let $Q(R)$ be the quotient division ring of $R$.
	Note that the $H$-module structure on $R$ has a unique extension to $Q(R)$ with respect to which $Q(R)$ becomes a left $H$-module algebra \cite[Theorem 2.2]{SVO2006}.
	Since $R\#H$ is semiprime by \cite[Theorem 0.5]{SVO2006}, $R\#H$ has a semisimple quotient ring $Q(R\#H)$ which is isomorphic to $Q(R)\#H$.
	By Theorem \ref{group-acts-on-division-algebra} and Lemma \ref{group-graded-division-algebra}, $Q(R)^H \subset Q(R)$ is a Galois extension.
	Hence the quotient ring $Q(R\#H)$ is simple by \cite[Theorem 8.3.7]{Mon1993}.
	Then $R\#H$ is prime, and $\varphi$ is injective by Lemma \ref{varphi-injective-lemma}.
\end{proof}

\subsection{Auslander theorem for PI AS regular algebras}

In this subsection, our main results are stated and proved.
To prove Theorem \ref{Auslander-thm-for-PI-rings}, we need several lemmas.

\begin{lem}\cite{Mon1993}\label{invariant-subring-lemma}
	Let $H$ be a semisimple Hopf algebra, $R$ be a noetherian $H$-module algebra.
	Then $R^H$ is also noetherian, and $R$ is a finitely generated $R^H$-module.
\end{lem}

\begin{lem}\label{Z(R)-cap-R^H}
	Let $H$ be a semisimple Hopf algebra, and $R$ be a left $H$-module algebra with center $Z(R)$. Suppose that $Z(R)$ is noetherian and $R$ is a finitely generated $Z(R)$-module.  If $H = \kk G$ or $(\kk G)^*$ where $G$ is a finite group, then $R^H$ is a finitely generated module over $Z = Z(R) \cap R^H$.
\end{lem}
\begin{proof}
	If $H = \kk G$, then $Z(R) \cap R^G = Z(R)^G$. Hence $Z(R)$ is a finitely generated $Z$-module by Lemma \ref{invariant-subring-lemma}.
	
	Now assume that $H = (\kk G)^*$. Hence $R = \oplus_{g \in G} R_g$ is a $G$-graded algebra.
	By our assumption,  $Z(R)R_1$ is also a finitely generated $Z(R)$-module. Let $x_1, \dots, x_n \in R_1$ such that $Z(R)R_1 = \sum_{i=1}^n Z(R)x_i$.
	For any $x \in R_1$, there exists $z_i \in Z(R)$, $u_i \in R_1$ and $v_i \in \oplus_{g \neq 1} R_g$ such that $x = \sum_{i=1}^n z_ix_i$ and $z_i = u_i + v_i$.
	Since $R$ is $G$-graded and $x$ is a homogeneous element, $x = \sum_{i=1}^n u_ix_i$.
	For all $g \in G$ and $r \in R_{g}$, $ru_i = u_ir$ because $ru_i + rv_i = rz_i = z_ir = u_ir + v_ir$. This implies that $u_i \in Z(R)$ for all $i$. So $R^H = R_1 = \sum_{i=1}^n Zx_i$.
	Hence $R^H$ is finitely generated as a $Z$-module.
\end{proof}

Let $R$ be a noetherian prime ring with artinian simple quotient ring $Q$.
Then $R$ is called an order in $Q$ and an order $S$ in $Q$ is said to be {\it equivalent} to $R$ if there exist units $a, b, c, d \in Q$ such that $aRb \subseteq S$ and $cSd \subseteq R$.
And $R$ is called a {\it maximal order} if it is maximal within its equivalence class, that is, if $S$ is an order in $Q$ equivalent to $R$ and containing $R$, then $S$ must be equal to $R$.

The following result is due to \cite[Theorem 5.3.13 and Proposition 13.6.11]{MR2001}.

\begin{lem}\label{center-of-Maximal-order}
	Let $R$ be a prime noetherian ring which is a finitely generated module over its center $Z(R)$. If $R$ is a maximal order (in its quotient ring), then $Z(R)$ a noetherian normal domain.
\end{lem}

%



\begin{lem}\cite[Theorem 5.3.16]{MR2001}\label{hereditary-maximal-order-over-Dedekind domain}
	Let $R$ be a noetherian prime ring which is a finitely generated module over its center $Z(R)$.
	Suppose that $R$ is a maximal order.
	Then $R$ is a hereditary ring if its center $Z(R)$ is a Dedekind domain.
\end{lem}

\begin{lem}\label{Kd-Coker-gamma}
	Let $R$ be a prime noetherian algebra which is a finitely generated module over its center $Z(R)$.
	Let $H$ be a semisimple Hopf algebra acting on $R$ inner faithfully. Suppose that $H = \kk G$ or $(\kk G)^*$ where $G$ is a finite group.
	If $R$ is a maximal order, then $\Kdim_R \Coker \gamma \leq \Kdim R - 2$.
\end{lem}
\begin{proof}
	Let $A = R^H$, $Z = Z(R) \cap A$ and $\p$ be a height one prime ideal of $Z$.
	Since $R$ is a maximal order over its center $Z(R)$, it follows that $Z(R)$ is a noetherian normal domain by Lemma \ref{center-of-Maximal-order}.
	Since $Z$ is a subring of $Z(R)$ such that $Z(R)$ is a finitely generated $Z$-module by Lemma \ref{Z(R)-cap-R^H}, the localization $Z(R)_{\p} := T^{-1}Z(R)$ of $Z(R)$ is a Dedekind domain, where $T = Z \setminus \p$.
	Let $R_{\p} := T^{-1}R$. Since $R_{\p}$ is also a maximal order by \cite[Theorem 11.1]{Rei2003}, $\gld R_{\p} \leq 1$ by Lemma \ref{hereditary-maximal-order-over-Dedekind domain}.
	
	According to \cite[Theorem 1.1]{Liu2005}, $\gld R_{\p}\#H = \gld R_{\p} \leq 1$.
	Since
	$$A_{\p} = T^{-1} A = (1\#t)(R_{\p} \# H)(1\#t),$$
	the global dimension of $A_{\p}$ is no more than one by \cite[Proposition 7.8.9]{MR2001}. Since $R_{\p}$ is a reflexive module over $A_{\p}$, $R_{\p}$ is a finitely generated projective $A_{\p}$-module by \cite[Corollary 2.13]{QWZ2019}.
	Notice that the localization map $(\gamma_R)_{\p} = T^{-1}\gamma_R$ of $\gamma_R$ can be seen canonically as
	$$\gamma_{R_{\p}}: R_{\p} \otimes_{A_{\p}} \Hom_{A_{\p}}(R_{\p}, A_{\p}) \To \End_{A_{\p}}(R_{\p}), \qquad x \otimes f \mapsto \big(y \mapsto xf(y) \big).$$
	So $(\gamma_R)_{\p}$ is bijective for any $\p \in \Spec Z$ with $\mathrm{ht} \p \leq 1$.
	It follows that
	$$\mathrm{Kdim}_{Z} \Coker \gamma \leq \mathrm{Kdim} Z - 2.$$
	By \cite[Corollary 6.4.13]{MR2001}, we know that
	$$\mathrm{Kdim}_{R} \Coker \gamma = \mathrm{Kdim}_{Z} \Coker \gamma, \; \mathrm{Kdim} Z = \mathrm{Kdim} R.$$
	This completes the proof.
\end{proof}


The CM and Auslander properties of the AS regular algebras have been studied in \cite{ASZ1998, SZ1994, YZ1999, Zha1997}.

\begin{thm}\cite[Theorem 1.1 and Corollary 1.2]{SZ1994}\label{Staf-Zhang-thm}
	Let $R$ be a noetherian connected graded PI ring. If $R$ has finite injective dimension $d$, then $R$ is Auslander-Gorenstein and CM with $\GKdim R = \Kdim R = d$. If $R$ has finite global dimension, then 
	
	(1) $R$ is a domain and a maximal order in its quotient division ring,
	
	(2) $R$ is finitely generated as a module over its center $Z(R)$.
\end{thm}

As well known, the smash product $B(:= R\#H)$ is a Frobenius extension of $R$, that is, $B$ is a finitely generated projective $R$-module and $B \cong \Hom_R(B, R)$ as $R$-$B$-bimodules.
\begin{lem}\label{Frob-Aus-Goren}\cite[Section 5.4]{AB2007}
	Let $R \subseteq B$ be a Frobenius extension. Then $j_R(M) = j_B(M)$ for any finitely generated $B$-module $M$. If $R$ is noetherian Auslander Gorenstein, then so is $B$.
\end{lem}

\begin{thm}\label{Auslander-thm-for-PI-rings}
	Let $R$ be a noetherian AS regular PI algebra, and $H$ be a semisimple Hopf algebra acting on $R$ homogeneously and inner faithfully. Suppose that $H = \kk G$ or $(\kk G)^*$ where $G$ is a finite group. 
	If the homological determinant of the $H$-action on $R$ is trivial, then the Auslander map $\varphi: R\#H \to \End_{R^H}(R)$ is bijective.
\end{thm}
\begin{proof}
	In order to apply Theorem \ref{main-thm-0} we need to verify the conditions (a), (b) and (c) in Proposition \ref{main-prop-0}.
	
	(a) It follows from Corollary \ref{Auslander-map-injective}.
	
	(b) It follows from Lemma \ref{nu-bijective}.

	(c) Since $R$ is a maximal order which is a finitely generated module over its center by Theorem \ref{Staf-Zhang-thm}, it follows that $\Kdim_R \Coker \gamma \leq \Kdim R - 2$ by Lemma \ref{Kd-Coker-gamma}. Hence $j_R(\Coker \gamma) \geq 2$ because $R$ is CM. 
	Therefore, $j_{R\#H}(\Coker \gamma) = j_R(\Coker \gamma) \geq 2$ by Lemma \ref{Frob-Aus-Goren}.
	
	Then the conclusion follows from Theorem \ref{main-thm-0}.
\end{proof}


We now give a few examples.

\begin{ex}
	Assume that $R$ is the $(-1)$-skew polynomial ring
	$$\kk_{-1}[x_1, \dots, x_d] := \kk\langle x_1, \dots, x_d \rangle / (x_ix_j+x_jx_i \mid i < j).$$
	Gaddis, Kirkman, Moore and Won proved that the Auslander map is an isomorphism for any finite subgroup of the symmetric group $\mathfrak{S}_d$ \cite[Theorem 3]{GKMW2019}.
	As we know, the automorphism group of $R$ is $(\kk^{\times})^d \rtimes \mathfrak{S}_d$ and $\hdet: \GrAut(R) \to \kk^{\times}$ is given by $\hdet(\lambda_1, \dots, \lambda_d; \sigma) = \prod_{i=1}^{d}\lambda_i$ (see \cite[Lemma 1.12]{KKZ2014}). Hence
	$$\SL(R) = \{ (\lambda_1, \dots, \lambda_d; \sigma) \mid \prod_{i=1}^{d}\lambda_i = 1, \lambda_i \in \kk, \sigma \in \mathfrak{S}_d \} \supseteq \mathfrak{S}_d,$$
	and the Auslander map is  an isomorphism for any finite subgroup of $\SL(R)$ by Theorem \ref{Auslander-thm-for-PI-rings}.
\end{ex}

\begin{ex}
	For any $\alpha \in \kk$ and $\beta \in \kk^{\times}$, the algebra $A(\alpha, \beta)$ over $\kk$ with generators $d, u$ and the defining relations
	$$d^2u = \alpha dud + \beta ud^2, \qquad du^2 = \alpha udu + \beta u^2d$$
	is a down-up algebra. 
	Bao, He and Zhang proved that the Auslander map is an isomorphism for any finite group of graded automorphisms when $\beta \neq -1$, and also for the case $A(2, -1)$ \cite[Theorem 0.6]{BHZ2018}.
	Let $\Delta = \sqrt{\alpha^2 + 4 \beta}$, $\omega_1 = \frac{\alpha - \Delta}{2}$, and $\omega_2 = \frac{\alpha + \Delta}{2}$.
	Then $A(\alpha, \beta)$ is a PI algebra if and only if $\Delta \neq 0$, and $\omega_1, \omega_2$ are roots of unity \cite[Theorem 1.3]{Zha1999}. 
	According to \cite[Proposition 1.1]{KK2005}, the graded automorphism group of $A(\alpha, -1)$ is
	$$\left\{ \begin{pmatrix} a_{11} & 0 \\ 0 & a_{22} \end{pmatrix}, \begin{pmatrix} 0 & a_{12} \\ a_{21} & 0 \end{pmatrix} \bigg| a_{ij} \in \kk^{\times} \right\}.$$
%
%
%
	By some computations, we can see that
	$$\SL(A(\alpha, -1)) = \left\{ \begin{pmatrix} a & 0 \\ 0 & \pm a^{-1} \end{pmatrix}, \begin{pmatrix} 0 & a \\ \pm a^{-1} & 0 \end{pmatrix} \bigg| a \in \kk^{\times} \right\}.$$
	
	Now assume that $\omega (\neq \pm 1)$ is a root of unity. Then $A(\omega + \omega^{-1}, -1)$ is a PI AS regular algebra. 
	By Theorem \ref{Auslander-thm-for-PI-rings}, we can see the Auslander map is an isomorphism for any finite subgroup of $\SL(A(\alpha, \beta))$ when $\alpha = \omega + \omega^{-1}$ and $\beta = -1$.
\end{ex}

\section{Applications to noncommutative resolutions}

In this section, some applications to noncommutative crepant (or quasi-) resolutions are indicated.
Let's recall the definitions of noncommutative resolutions.

\begin{defn}\cite{VdBer2004}\label{NCCR-defn}
	Let $A$ be a noetherian Gorenstein normal domain. A {\it noncommutative crepant resolution} (or {\it NCCR} for short) of $A$ is an algebra $\End_A(M)$ where $M$ is a reflexive $A$-module and where $\End_A(M)$ has finite global dimension and is a maximal CM $A$-module.
\end{defn}

	Let $n$ be a nonnegative integer, $A$ and $B$ be two $\N$-graded algebras.
	Two $\Z$-graded $B$-$A$-bimodules $X$, $Y$ are called $n$-isomorphic, denoted by $X \cong_{n} Y$, if there exists a $\Z$-graded $B$-$A$-bimodules $P$ and $\Z$-graded bimodule morphisms $f: X \to P$ and $g: Y \to P$ such that both the kernel and cokernel of $f$ and $g$ have GK-dimension no more than $n$.

\begin{defn}\cite[Definition 0.5]{QWZ2019}
	Let $A$ be a noetherian $\N$-graded algebra with $\GKdim(A) = d (\geq 2) \in \N$. If there exists a noetherian locally finite $\N$-graded Auslander regular CM algebra $B$ with $\GKdim(B) = d$ and two $\Z$-graded bimodules $_BM_A$ and $_AN_B$, finitely generated on both sides, such that
	$$M \otimes_A N \cong_{d-2} B, \text{ and } N \otimes_B M \cong_{d-2} A$$
	as $\Z$-graded bimodules, then the triple $(B, M, N)$ or simply the algebra $B$ is called a {\it noncommutative quasi-resolution} (or {\it NQR} for short) of $A$.
\end{defn}

Many examples of NQRs are produced by the Auslander theorem.

\begin{thm}\cite[Proposition 8.3 and Example 8.5]{QWZ2019}\label{NQR-thm}
	Let $R$ be a noetherian connected graded Auslander regular CM algebra with $\GKdim(R) = d \geq 2$, and $H$ be a semisimple Hopf algebra acting on $R$ homogeneously and inner faithfully with integral $t$ such that $\vep(t) = 1$. If the Auslander map is bijective, then $(B, Be, eB)$ is a NQR of $R^H$, where $B = R\#H$ and $e =1\#t$.
\end{thm}

A connected graded AS regular algebra $R$ is called {\it Calabi-Yau} if the Nakayama automorphism of $R$ is the identity map.

\begin{lem}\cite{WZ2022}\label{normal-hdet}
	Let $R$ be a noetherian connected graded Calabi-Yau algebra.
	Suppose that there exists a normal regular element $z \in R$ and $\sigma \in \Aut(R)$ such that $zx = \sigma(x)z$ for all $x \in R$.
	Then $\hdet(\sigma) = 1$.
\end{lem}

By abuse of notation, the smash product of $R$ by a group algebra $\kk G$ is denoted by $R\#G$.

\begin{thm}\label{NCCR-thm}
	Let $R$ be a noetherian AS regular PI algebra.
	\begin{enumerate}
		\item Let $G$ be a finite subgroup of $\SL(R)$. Then $R\#G$ is a NQR of $R^G$. If further, $R^G$ is commutative, then $R\#G$ is a NCCR of $R^G$.
		\item Let $Z$ be the center of $R$. Suppose that $R$ is Calabi-Yau. 
		\begin{itemize}
			\item[(i)] Then $\GrAut_{Z}(R)$ is a finite subgroup of $\SL(R)$.
			\item[(ii)] If $R^{\GrAut_{Z}(R)} = Z$, then $\End_Z(R)$ is a NCCR of $Z$.
		\end{itemize}
	\end{enumerate}
\end{thm}
\begin{proof}
	(1)
	The first assertion follows from that Theorem \ref{Auslander-thm-for-PI-rings} and \ref{NQR-thm}.
	By Lemma \ref{dualizing complex}, $R^G$ is a Gorenstein domain such that $R$ is a maximal CM $R^G$-module.
	Recall that $R\#G \cong \End_{R^G}(R)$ has finite global dimension by \cite[Theorem 1.1]{Liu2005}, and that $R^G = Z(\End_{R^G}(R))$ is normal by \cite[Lemma 2.2]{SVdB2008}.
	Since $R \cong \Hom_{R^G}(R, R^G)$ as $R^G$-modules, $R$ is a reflexive $R^G$-module.
	Thus $\End_{R^G}(R)$ is a maximal order.
	So $R\#G$ is a NCCR of $R^G$ by the definition.
	
	(2)
	(i) Let $K$ be the fraction field of $Z$. Then $Q:= R \otimes_Z K$ is a central simple $K$-algebra by Posner's theorem \cite[Theorem 13.6.5]{MR2001}. Let $\sigma \in \GrAut_{Z}(R)$. Since $\sigma$ extends to a $K$-algebra automorphism of $Q$, it follows that $\sigma$ is inner, that is, there exists a nonzero element $z \in Q$ such that $\sigma(q) = zqz^{-1}$ for all $q \in Q$. Without loss of generality, we will assume that $z \in R$. Clearly, $z$ is a normal regular element of $R$. Then by Proposition \ref{normal-hdet}, $\hdet(\sigma) = 1$. Hence $\GrAut_{Z}(R) \subseteq \SL(R)$.

For any $\sigma \in \GrAut_{Z}(R)$, we claim that $\sigma^{n!} = \id_R$ where $n = \dim_K(Q)$.
As proved above, there exists a normal regular element $z \in R$ such that $\sigma(r)z = zr$ for all $r \in R$.
Since $z$ is integral over $Z$, there exists an integer $k \leq n$ and $a_i \in Z$ such that
$$z^k + a_{k-1}z^{k-1} + \cdots + a_1z + a_0 = 0.$$
Without loss of generality, we will assume that $k$ is minimum.
Since $R$ is a noetherian connected graded ring, $R_m$ is finite dimensional over $\Bbbk$ for all $m \geq 0$.
Recall that $\Bbbk$ is an algebraically closed field of characteristic zero.
For any $m \geq 0$, if $R_m$ is not a semisimple $\Bbbk \langle \sigma \rangle$-module, then there exist $x, y \in R_m \setminus \{ 0 \}$ such that $\sigma(y) = \lambda y + x$ and $\sigma(x) = \lambda x$ for some $\lambda \in \Bbbk \setminus \{ 0 \}$.
Notice that for any $i \geq 1$,
$$0 = (z^k + a_{k-1}z^{k-1} + \cdots + a_1z + a_0)x^i = x^i((\lambda^i z)^k + a_{k-1}(\lambda^i z)^{k-1} + \cdots + a_1(\lambda^i z) + a_0).$$
Hence $\lambda^i z$ is also a solution of the equation
\begin{equation}\label{z-equation}
	X^k + a_{k-1}X^{k-1} + \cdots + a_1X + a_0 = 0.
\end{equation}
Since $Z[z]$ is commutative subring of a domain $R$, the equation \eqref{z-equation} has at most $k$ distinct roots in $Z[z]$.
Hence there exists an integer $l \leq k$ such that $\lambda^l = 1$.
Let $y' = x^{l-1}y$ and $x' = \lambda^{l-1}x^l$. Then $\sigma(y') = y' + x'$ and $\sigma(x') = x'$.
Thus we have
\begin{align*}
	0 = & (z^k + a_{k-1}z^{k-1} + \cdots + a_1z + a_0)y' \\
	= & (y' + kx')z^k + (y' + (k-1)x')a_{k-1}z^{k-1} + \cdots + (y' + x')a_1z + y'a_0 \\
	= & y' (z^k + a_{k-1}z^{k-1} + \cdots + a_1z + a_0) + x' (kz^k + (k-1)a_{k-1}z^{k-1} + \cdots + a_1z) \\
	= & x' (kz^{k-1} + (k-1)a_{k-1}z^{k-1} + \cdots + a_1) z.
\end{align*}
which is a contradiction since $R$ is a domain.
Hence $R_m$ is a semisimple $\Bbbk \langle \sigma \rangle$-module for any $m \geq 0$.
Therefore, for any $x \in R_m$, there exist $\lambda_1, \dots, \lambda_s \in \kk \setminus \{ 0 \}$ and $x_1, \dots, x_s \in R_m$ such that
$$x = x_1 + \cdots + x_s, \text{ and } \sigma(x_i) = \lambda_i x_i \text{ for all } i.$$
By the above proof, there exist positive integers $l_1, \dots, l_s \leq n$ such that $\lambda_i^{l_i} = 1$ for all $i$.
It follows that $\sigma^{n!} = \id_R$.


	Since $G = \GrAut_{Z}(R)$ is a subgroup of $\Aut_{K}(Q) (\subseteq \mathrm{GL}_n(K))$ with finite exponent, it follows that $G$ is a finite group by Burnside theorem \cite[8.1.11]{Rob1995}.
	
	(ii)
	The conclusion follows immediately from (1).
\end{proof}

According to Theorem \ref{NCCR-thm} (2) (ii), we establish a criterion for the center of a noetherian PI connected graded Calabi-Yau algebra to have a NCCR.
Now let's consider the following example.

\begin{ex}
	Let $\{ p_{ij} \in \kk^{\times} \mid 1 \leq i < j \leq d \}$ be a set of roots of unity, and set $p_{ji} = p^{-1}_{ij}$ and $p_{ii} = p_{jj}$ for all $i < j$. The skew polynomial ring is defined to be the algebra generated by $x_1, \dots, x_d$ subject to the relations $x_jx_i = p_{ij}x_ix_j$ for all $i < j$, and is denoted by $\kk_{p_{ij}}[x_1, \dots, x_d]$.
	
	Assume that $\prod^d_{j=1} p_{ji} = 1$ for all $i = 1, \dots, d$. Hence the Nakayama automorphism is the identity map by \cite[Proposition 4.1]{LWW2014}.
	Let $Z$ be the center of $\kk_{p_{ij}}[x_1, \dots, x_d]$. It is easy to see that $G := \GrAut_Z(\kk_{p_{ij}}[x_1, \dots, x_d])$ is generated by $\sigma_1, \cdots, \sigma_d$, where $\sigma_i$ is defined by $\sigma_i(x_j) = p_{ij}x_j$.
	Since $Z = \kk_{p_{ij}}[x_1, \dots, x_d]^{G}$, it follows that $Z$ has a NCCR by Theorem \ref{NCCR-thm}.
	In fact, $G$ can be seen as a subgroup of the automorphism group of polynomial ring $\kk[x_1, \dots, x_d]$. It is well known that the skew group algebra $\kk[x_1, \dots, x_d]\#G$ is a NCCR of the invariant subring $\kk[x_1, \dots, x_d]^G$ (see \cite[Example 1.1]{VdBer2004}).
\end{ex}

%


Obviously, not all of PI connected graded Calabi-Yau algebra satisfies the assumption of Theorem \ref{NCCR-thm} (2) (ii). See the following examples.

\begin{ex}
	Let $R$ be the down-up algebra $A(0, -1)$.
	By \cite[E 1.5.6]{LMZ2017}, the Nakayama automorphism of $R$ is the identity map.
	The center of $R$ is
	$$Z = \kk[d^4, u^4, dudu+udud, (du+\sqrt{-1}ud)^4, (du-\sqrt{-1}ud)^4].$$
	It is not difficult to see that
	$$\GrAut_{Z}(R) = \left\{ \begin{pmatrix} (\sqrt{-1})^i & 0 \\ 0 & \pm (\sqrt{-1})^{4-i} \end{pmatrix} \right\},$$
	and the invariant subring $R^{\GrAut_{Z}(R)} (= \kk[d^4, u^4, d^2u^2, dudu, udud] \neq Z)$ is a commutative algebra.
	Then  $R^{\GrAut_{Z}(R)}$ has a NCCR by Theorem \ref{NCCR-thm}.
\end{ex}

There also exists a connected graded Calabi-Yau algebra whose center doesn't have a NCCR (see \cite[Example 9.3]{SVdB2017}).

\begin{ex}
	With $\mathcal{P}_{m,n}$ we denote the graded polynomial algebra
	$$\kk[x_{ij}(l) \mid 1 \leq i, j \leq n; 1 \leq l \leq m],$$
	with $\deg(x_{ij}(l)) = 1$.
	The $\kk$-subalgebra of the matrix ring $M_n(\mathcal{P}_{m,n})$ generated by the matrices
	$$X_l = (x_{ij}(l))_{i, j}, \qquad \text{where } 1 \leq l \leq m,$$
	is called the ring of $m$ generic $n \times n$ matrices $\mathbb{G}_{m, n}$.
	The $\kk$-subalgebra of $M_n(\mathcal{P}_{m,n})$ generated by $\mathbb{G}_{m, n}$ and $\Tr(\mathbb{G}_{m, n})$ is the trace ring of $m$ generic $n \times n$ matrices and is denoted by $\mathbb{T}_{m, n}$.
	Note that both $\mathbb{G}_{m, n}$ and $\mathbb{T}_{m, n}$ are connected graded subalgebra of $M_n(\mathcal{P}_{m,n})$.
	
	
	By \cite[Example 9.3]{SVdB2017}, $\mathbb{T}_{3, 2}$ is a twisted NCCR of its center $Z_{3,2}$, but $Z_{3,2}$ doesn't have a NCCR. Hence
	By \cite[Example 9.3]{SVdB2017}, the center $Z_{3,2}$ of the Calabi-Yau algebra $\mathbb{T}_{3, 2}$ doesn't have a NCCR. Hence
	$$\mathbb{T}_{3, 2}^{\GrAut_{Z_{3,2}} (\mathbb{T}_{3, 2})} \neq Z_{3,2}.$$
\end{ex}

\section*{Acknowledgments}
The author thanks the referee for the valuable suggestions which improved the presentation of this article quite a lot. I'm also grateful to Professor Ellen Kirkman and James Zhang for helpful comments on a draft of this paper. 

\bibliographystyle{siam}
\bibliography{References}

\end{document}